\documentclass[12pt,reqno]{amsart}
\usepackage{jrmacros}
\usepackage{bbm}
\usepackage{mathtools}
\mathtoolsset{showonlyrefs}
\usepackage{enumitem}
\usepackage{tikz-cd}

\numberwithin{equation}{section}

\makeatletter
\newtheorem*{rep@theorem}{\rep@title}
\newcommand{\newreptheorem}[2]{%
\newenvironment{rep#1}[1]{%
 \def\rep@title{#2 \ref*{##1}}%
 \begin{rep@theorem}}%
 {\end{rep@theorem}}}
\makeatother

\theoremstyle{plain}
\newtheorem{theorem}{Theorem}[section]
\newtheorem{proposition}[theorem]{Proposition}
\newtheorem{corollary}[theorem]{Corollary}

\newtheorem{conj}{Conjecture}
\newreptheorem{conj}{Conjecture}
\newtheorem{thm}{Theorem}

\newtheorem*{theorem*}{Theorem}
\newtheorem*{proposition*}{Proposition}
\newtheorem*{corollary*}{Corollary}
\newtheorem*{lemma*}{Lemma}
\newtheorem*{conjecture*}{Conjecture}

\theoremstyle{definition}
\newtheorem{definition}[theorem]{Definition}

\newtheorem*{definition*}{Definition}
\newtheorem*{example*}{Example}
\newtheorem*{question*}{Question}
\newtheorem*{philosophy*}{Philosophy}

\theoremstyle{remark}
\newtheorem{remark}[theorem]{Remark}

\newtheorem*{remark*}{Remark}

\def\A{\cR}
\def\gal{\mathrm{Gal}}

\def\sp{\mathrm{Spec}\,}

\def\au{\underline{Aut}^\otimes}
\def\per{per_{\A}}
\def\Pd{\cP^{\mathfrak{dr}}}
\def\Ai{\widehat{\A}}
\def\zd{\zeta^{\mathfrak{dr}}}
\def\p{\mathbf{p}}
\def\fil{\mathrm{Fil}}

\def\L{\mathbb{L}}

\def\Zd{\zd}
\def\dr{\mathfrak{dr}}
\def\Pdh{\widehat{\cP}^{\dr}}
\def\per{\mathrm{per}}

\def\du{*}

\def\uni{sequential}
\def\Uni{Sequential}

\def\Lm{\L^{\fm}}
\def\Ld{\L^{\dr}}

\def\og{\mathbf{Og}(\Q)}
\def\spa{\mathbf{Span}(\Q)}

\def\zm{\zeta^{\fm}}

\newcommand\MM{\cA\cM_\Q}
\newcommand\hper{\per_{\Ai}}

\newcommand\iso{\xrightarrow{\sim}}
\newcommand\MT{\mathcal{MT}}

\begin{document}


\title{Sequential periods of the crystalline Frobenius}
\author{Julian Rosen}
\date{\today}
\maketitle

\begin{abstract}
There is a notion of $p$-adic period coming from the crystalline Frobenius automorphism of the de Rham cohomology of an algebraic variety. In this paper, we consider sequences of $p$-adic periods, one for each prime. We study the sequences using motivic periods, and we formulate an analogue of the Grothendieck period conjecture.
\end{abstract}

%
%
\section{Introduction}

Let $X$ be a smooth algebraic variety over $\Q$. The algebraic de Rham cohomology $H_{dR}(X)=\bigoplus_n H^n_{dR}(X)$ 
is a finite-dimensional vector space over $\Q$. For every sufficiently large prime $p$, there is a distinguished $\Q_p$-linear automorphism
\begin{equation}
\label{deff}
F_{p,X}:H_{dR}(X)\otimes \Q_p\xrightarrow{\sim} H_{dR}(X)\otimes \Q_p,
\end{equation}
the \emph{crystalline Frobenius map}, coming from the absolute Frobenius endomorphism of the reduction of an integral model of $X$ modulo $p$ (see \cite{Ked09}). If we choose a $\Q$-basis for $H_{dR}(X)$, we can represent $F_{p,X}$ as a square matrix with entries in $\Q_p$, and $\Q$-linear combinations of matrix entries are called \emph{$p$-adic periods}\footnote{This is one of several different meanings of the term ``$p$-adic period'' that have appeared in the literature.} of $X$. Note that the $\Q$-span of matrix entries in independent of the choice of basis.

It makes sense to talk about ``the same'' $p$-adic period for different $p$. We consider sequences of $p$-adic periods, one for each sufficiently large $p$, living in the ring
\[
\A:=\frac{\prod_p\Q_p}{\bigoplus_p\Q_p}.
\]
An element of $\A$ is a prime-indexed sequence $(a_p)_p$, with $a_p\in\Q_p$, and two sequences are equal in $\A$ if they agree for all sufficiently large $p$. The maps $F_{p,X}$ (for all large $p$ at once) assemble to give an $\A$-linear automorphism
\begin{equation}
\label{eqFA}
F_{\A,X}:H_{dR}(X)\otimes\A\xrightarrow{\sim} H_{dR}(X)\otimes\A,
\end{equation}
which we call the \emph{\uni{} Frobenius map}. Like before, if we choose a $\Q$-basis for $H_{dR}(X)$, we can represent $F_{\A,X}$ as a square matrix with entries in $\A$.
The following definition is new.

\begin{definition}
A \emph{\uni{} period} (or \emph{$\A$-valued period}) of $X$ is a $\Q$-linear combination of matrix entries for $F_{\A,X}$. We denote the set of \uni{} periods of $X$ by $\cP_\A(X)\subset\A$.
\end{definition}


We compute several examples of \uni{} periods in Sec.\ \ref{secex}.

\begin{remark}
 \Uni{} periods are analogous to the usual periods of $X$, which are complex numbers arising from the comparison between de Rham and Betti cohomology. The analogy is explained in \ref{secmot}.
\end{remark}

\subsection{Results}
In this paper, we develop a motivic theory of \uni{} periods, and we formulate an analogue of the Grothendieck period conjecture for \uni{} periods. Our main results are several consequences of the conjecture. The conjecture has a weak form (Conjecture \ref{conup}) and a strong form (Conjecture \ref{conups}). Conjecture \ref{conup} predicts an answer to the question: When are two \uni{} periods equal? Conjecture \ref{conups} predicts an answer to the question: When are two \uni{} periods congruent modulo $p^n$ for all sufficiently large $p$?

Our conjectures are assertions about a category of (mixed) motives. Suppose $\cM$ is a Tannakian category of motives over $\Q$ (the precise properties we need are given in Definition \ref{defmot}), equipped with a fibre functor $\omega_{dR}$, the de Rham realization. Write $G_{dR}$ for the affine group-scheme $\au(\omega_{dR})$. The coordinate ring of $G_{dR}$ is a commutative $\Q$-algebra $\Pd$, called the \emph{ring of de Rham  periods} of $\cM$ (see \cite{Bro14}, \S2).
If $\cM$ satisfies the conditions of Definition \ref{defmot}, the crystalline Frobenius (for all sufficiently large $p$ at once) gives a distinguished functorial automorphism $F_{\A,M}$ of $\omega_{dR}(M)\otimes\A$ for $M\in\cM$. These automorphisms determine
a ring homomorphism
\[
\per_\A:\Pd\to\A,
\]
which we call the \emph{\uni{} period map}. The image of $\per_\A$ is the $\Q$-span of all matrix entries for $F_{\A,M}$ (for all $M\in\cM$).

The weak form of our conjecture is the following assertion about $\cM$.
\begin{conj}[\Uni{} period conjecture]
\label{conup}
The map $\per_\A$ is injective.
\end{conj}

\noindent By contrast, a similar $p$-adic period map $\Pd\to\Q_p$ is typically not injective (see Sec.\ \ref{ssfp}).
\smallskip

For $X$ smooth and projective, it is known that the restriction of $F_{p,X}$ to the $n$-th level of the Hodge filtration on $H_{dR}(X)$ has matrix entries divisible by $p^n$ when $p$ is sufficiently large. The strong form of our conjecture is a converse to this divisibility property. It asserts that if a \uni{} period is divisible by $p^n$ for all sufficiently large $p$, then that period is motivically equivalent (in a precise sense) to something coming from the $n$-th level of the Hodge filtration.

To formulate this conjecture precisely, we use the decreasing filtration $\fil^\bullet$ on $\A$ given by
\begin{equation}
\label{eqdeffil}
\fil^n\A:=\big\{(a_p)_p:\liminf_p v_p(a_p)\geq n\big\}.
\end{equation}
There is also a decreasing filtration $\fil^\bullet$ on $\Pd$, coming from the Hodge filtration on algebraic de Rham cohomology.
The strong form of our conjecture is the following assertion about $\cM$.
\begin{conj}[Strong \uni{} period conjecture]
\label{conups}
The \uni{} period map $\per_\A$ is an embedding of filtered algebras, i.e.\ for every integer $n$,
\[
\per_\A^{-1}\big(\fil^n\A)=\fil^n\Pd.
\]
\end{conj}

We give several consequences of Conjectures \ref{conup} and \ref{conups}. We list some of them here.
The first two results are consequences of {Conjecture \ref{conup}}.

\begin{thm}[Theorem \ref{thcyc}]
Assume the standard conjectures on algebraic cycles, and suppose that Conjecture \ref{conup} holds for the category of pure numerical motives over $\Q$. Let $X$ be a smooth projective variety over $\Q$. If $\alpha$, $\beta\in H_{dR}(X)$ satisfy $F_{p,X}(\alpha)=p^n \beta$ for all but finitely many $p$, then $\alpha$ is algebraic.
\end{thm}

\begin{thm}[Corollary \ref{thog}]
Let $\cM$ be a category of motives satisfying Definition \ref{defmot}. If Conjecture \ref{conup} holds for $\cM$, then the Ogus realization of $\cM$ is a fully faithful embedding.
\end{thm}

\noindent The following results give consequences of Conjecture \ref{conups} applied to the categories $\MT(\Z)$ and $\MT(\Q)$ of mixed Tate motives over $\Z$ resp.\ $\Q$.

\begin{thm}[Theorem \ref{thnwi}]
If Conjecture \ref{conups} holds for $\MT(\Q)$, then for every rational number $r\neq 0,\pm1$, there exist infinitely many $p$ for which $r^{p-1}\not\equiv 1\mod p^2$.
\end{thm}

\begin{thm}[Theorem \ref{thnwo}]
If Conjecture \ref{conups} holds for $\MT(\Z)$, then for every odd $k\geq 3$, there exist infinitely many primes $p$ for which $p\nmid B_{p-k}$.
\end{thm}
\begin{thm}[Theorem \ref{thfmzv}]
Assume that the Grothendieck period conjecture (Conjecture \ref{conper} below) holds for $\MT(\Z)$. Then the truth of Conjecture \ref{conups} for $\MT(\Z)$ is equivalent to Kaneko-Zagier's conjecture on finite multiple zeta values.
\end{thm}

\subsection{Outline}
In Section \ref{secex} we compute several examples of \uni{} periods. Section \ref{secmot} gives the properties necessary for a category of motives in order to  construct the \uni{} period map, and to state Conjecture \ref{conup}. In Section \ref{secval}, we consider divisibility properties of \uni{} periods and state Conjecture \ref{conups}.

Sections \ref{secpcc}, \ref{appp}, and \ref{appc} are independent of one another, and can be read in any order. In Section \ref{secpcc}, we deduce several consequences of Conjectures \ref{conup} and \ref{conups}. Sections \ref{appp} and \ref{appc} describe two variations of \uni{} periods. The first variation involves reducing \uni{} periods modulo $p$, and this variation includes Kaneko-Zagier's finite multiple zeta values as a special case. The second variation involves allowing uniformly convergent infinite sums of \uni{} periods.

In Section \ref{secMTZ}, we describe the \uni{} periods of the category of mixed Tate motives over $\Z$.

\subsection{Acknowledgements}

We thank Bhargav Bhatt and Kartik Prasanna for helpful discussions. We thank Jeffrey Lagarias for comments on this manuscript.

%
%
\section{Examples}
\label{secex}

We compute some examples of \uni{} periods.
\subsection{The projective line}
\label{ssp1}
The vector space $H^2_{dR}(\mathbb{P}^1)$ is $1$-dimensional, and $F_p$ acts on $H^2_{dR}(\mathbb{P}^1)$ by multiplication by $p$. This shows that 
\[
\p:=(p)_p\in\A
\]
is a \uni{} period. The complex number $2\pi i$ is a period of $H^2(\mathbb{P}^1)$, so we view $\p$ as the \uni{} analogue of $2\pi i$. Note that $\p$, like $2\pi i$, is transcendental over $\Q$.

\subsection{Point counts}
Let $X$ be a smooth projective variety over $\Q$. Choose an integral model $\tilde{X}$ of $X$, and consider the sequence
\begin{equation}
\label{eqpc}
\bigg(\#\tilde{X}(\F_{p})\bigg)_p\in\A,
\end{equation}
which is independent of  the choice of $\tilde{X}$. The Lefschetz fixed point formula expresses $\#\tilde{X}(\F_p)$ as an alternating sum of traces of $F_{p,X}$ on cohomology of $X$, so \eqref{eqpc} is a \uni{} period of $X$.

\subsection{The $p$-adic logarithm}
For $r\in\Q_{>0}$, the logarithm $\log(r)\in\R$ is a period. The $p$-adic logarithm is defined by
\[
\log_p(x)=\sum_{n\geq 1}(-1)^{n+1}\frac{(x-1)^n}{n},
\]
which converges $p$-adically if $v_p(x-1)\geq 1$. For fixed $r$, the element
\[
\big(\log_p(r^{p-1})\big)_p\in\A
\]
is a \uni{} period (see Sec.\ \ref{sswie}).

\subsection{The $\ell$-adic Frobenius}
Let $X$ be a smooth projective variety over $\Q$, and fix an integer $n\geq 0$ and a prime $\ell$. The for $p\neq\ell$,  the geometric Frobenius at $p$ acting on the $\ell$-adic cohomology $H^n_{\ell}(X;\Q_\ell)$ has the same characteristic polynomial as $F_{p,X}$. This implies the sequence of characteristic polynomials of the Frobenius at $p$ acting on $H^n_{\ell}$, as $p$ varies, is a polynomial with coefficients in $\cP_\A(X)$.


\subsection{Modular forms}
\label{ssmf}
Let $f(z)=\sum_{n\geq 0} a_n e^{2\pi inz}$ be a normalized cusp form of weight $k$ for $\Gamma_0(N)$, with $a_n\in\Q$. Then there is a variety $V$ over $\Q$ (a power of a family of elliptic curves over the modular curve $X_0(N)$) such that $a_p$ is the trace of the Frobenius at $p$ acting on an $\ell$-adic cohomology group of $V$, and it follows that the sequence $\big(a_p\big)_p\in\A$ is a \uni{} period.

%
%
\section{Motivic periods and period maps}
\label{secmot}

Conjectures \ref{conup} and \ref{conups} are formulated relative to a category of motives. We begin by listing the properties required of this category.

\begin{definition}
\label{defmot}
We say $\cM$ is a \emph{category of motives (over $\Q$)} if $\cM$ is a neutral Tannakian category (\cite{Del89}, Definition 2.19) over $\Q$ equipped with the following structure.
\begin{wea}
\item There is a fibre functor $\omega_{dR}$ (the de Rham realization), $M\mapsto M_{dR}$, equipped with a decreasing, separated, exhaustive filtration $\fil^\bullet$, the Hodge filtration. The Hodge filtration is compatible with the tensor product, and if $f:M\to N$ is a morphism in $\cM$, then the induced map $f_{dR}:M_{dR}\to N_{dR}$ satisfies $f_{dR}(\fil^n M_{dR})=f(M_{dR})\cap\fil^n N_{dR}$ for all $n$.
\item There is a fibre functor $\omega_{B}$ (the Betti realization), $M\mapsto M_B$, and a functorial $\C$-linear isomorphism
\begin{equation}
\label{eqdrb}
comp_M:M_{dR}\otimes\mathbb{C}\xrightarrow{\sim}M_B\otimes\mathbb{C}
\end{equation}
compatible with the tensor product.

\item \label{propnum} For each $M\in\cM$, there is a distinguished automorphism
\[
F_{p,M}:M_{dR}\otimes\Q_p\iso M_{dR}\otimes\Q_p,
\]
the crystalline Frobenius map, defined for all sufficiently large $p$. The crystalline Frobenius is functorial in $M$ (for large $p$) and compatible with the tensor product.
\end{wea}

\end{definition}
Definition \ref{defmot} holds, for example, for categories satisfying the axioms in \cite{Del89}, \S1.3.

\begin{remark}
The data of $F_{p,M}$ is equivalent to the data of an $\A$-linear automorphism $F_{\A,M}$ of $M_{dR}\otimes\A$. The \emph{Ogus category} $\og$ is defined to be the category whose objects are pairs $(V,F)$, where $V$ is a finite-dimensional vector space over $\Q$ and $F$ is an $\A$-linear automorphism of $V\otimes\A$. If $\cM$ is a category of motives, the \emph{Ogus realization} of $\cM$ is the functor $\cM\to\og$, $M\mapsto (M_{dR},F_{\A,M})$ (see \cite{And04}, Sec.\ 7.1.5).
\end{remark}

For the remainder of this section, we fix a category of motives $\cM$.

\subsection{Complex periods}
Suppose $M\in\cM$. If we choose $\Q$-bases for $M_{dR}$ and $M_B$, we can represent $comp_M$ as a square matrix with entries in $\C$. A (complex) \emph{period} of $M$ is a $\Q$-linear combination of matrix entries for $comp_M$.

For any two fibre functors $\omega$ and $\eta$ on $\cM$, the functor of $\Q$-algebras $R\mapsto\hom^\otimes(\omega\otimes R,\eta\otimes R)$ is representable by an affine pro-algebraic scheme $\underline{\hom}^\otimes(\omega,\eta)$ (\cite{Del82}, Theorem 3.2). The scheme $\underline{\hom}^\otimes(\omega_{dR},\omega_B)$ is called the \emph{torsor of periods} of $\cM$, and the coordinate ring $\cP^\fm$ of $\underline{\hom}^\otimes(\omega_{dR},\omega_B)$ is called the \emph{ring of motivic periods} of $\cM$. The comparison isomorphism \eqref{eqdrb} determines an element of $\underline{\hom}^\otimes(\omega_{dR},\omega_B)(\C)$, and evaluation at this point induces a ring homomorphism
\begin{equation}
\label{eqcper}
\per:\cP^{\fm}\to\C,
\end{equation}
the \emph{period map} (see \cite{Bro17}, \S1.2).

The image of $\per$ is the set of all periods of all objects of $\cM$. We write $\cP_{\C}(\cM)$ for the image of $\per$, and we call $\cP_{\C}(\cM)$ the \emph{ring of periods of $\cM$}.  The Grothendieck period conjecture for $\cM$ is the following assertion.

\setcounter{conj}{-1}
\begin{conj}[Period conjecture]
\label{conper}
The map $per$ is injective.
\end{conj}
Conjecture \ref{conper} is an algebraic independence statement for periods, and is expected to be difficult to resolve.
\smallskip

The ring $\Pm$ is a torsor\footnote{In fact $\Pm$ is a bitorsor for $\au(\omega_B)$ and $\au(\omega_{dR})$.} for the affine (pro-algebraic) group scheme $\au(\omega_B)$. If Conjecture \ref{conper} holds, then the action descends to an action of $\au(\omega_B)$ on $\cP_\C(\cM)$, and we get a ``Galois theory of periods''. For the category of Artin motives (that is, motives of $0$-dimensional varieties), Conjecture \ref{conper} is known to hold, and $\per$ maps $\Pm$ isomorphically onto $\Qb\subset\C$. The group $\au(\omega_B)$ is canonically isomorphic to $\gal(\Qb/\Q)$ (viewed as a profinite constant group scheme over $\Q$), and thus the Galois theory of periods for Artin motives is just the classical Galois theory (see \cite{Bro17}, \S5.1).

\subsection{\Uni{} periods}
The affine group scheme $G_{dR}:=\au(\omega_{dR})$ is called the \emph{de Rham Galois group} of $\cM$, and the coordinate ring $\Pd$ of $G_{dR}$ is called the \emph{ring of de Rham periods} of $\cM$. The automorphisms $F_{\A,M}$ determine an element $F_\A\in\Pd(\A)$. The following definition is new.
\begin{definition}
The \emph{\uni{} period map} is the ring homomorphism
\[
\per_{\A}:\Pd\to\A
\]
induced by evaluation at $F_{\A}$.
\end{definition}
Here $\A$ plays the same role as $\C$ plays for the complex period map \eqref{eqcper}.
We write $\cP_\A(\cM)\subset\A$ for the image of $per_{\A}$, and we call $\cP_\A(\cM)$ the \emph{ring of \uni{} periods of $\cM$}. Our weaker analogue of the period conjecture for \uni{} periods is the following assertion about $\cM$.

\begin{repconj}{conup}[\Uni{} period conjecture]
The map $per_{\A}$ is injective.
\end{repconj}
Conjecture \ref{conper} is equivalent to the statement that $\per_{\A}$ induces an isomorphism of algebras $\Pd\iso\cP_\A(\cM)$.
As with Conjecture \ref{conper}, the truth of Conjecture \ref{conup} would give a	Galois theory of \uni{} periods (see Sec.\ \ref{ssgal}). We expect that Conjecture \ref{conup} will be difficult to resolve.

\subsection{$p$-adic periods}
\label{ssfp}
Suppose $p$ is a prime of good reduction for every object of $\cM$, meaning that for every $M\in\cM$, $F_{p,M}$ is defined and is functorial in $M$. In this case we get a \emph{$p$-adic period map}
\[
\per_p:\Pd\to\Q_p.
\]
However, $\per_p$ is typically not injective. For example, $F_p$ acts on the de Rham realization of the Tate motive by multiplication by $p^{-1}$, which is rational, so $\per_p$ fails to be injective if the Tate motive is in $\cM$. It can also be shown that $\per_p$ is not injective for the category of Artin motives.

In the case of mixed Tate motives, a modified version of the $p$-adic period conjecture is expected to hold. For mixed Tate motives, the de Rham Galois group decomposes as a semi-direct product
\[
G_{dR}=\G_m\ltimes U_{dR},
\]
where $U_{dR}$ is pro-unipotent. If we write $A_{dR}$ for the coordinate ring of $U_{dR}$, then there is a quotient map $\Pd\surj A_{dR}$. Although the $p$-adic period map does not factor through $A_{dR}$, there is a rescaled version $\varphi_p:\Pd\to\Q_p$ that does factor through $A_{dR}$, and it is conjectured (\cite{Yam10}, Conjecture 4) that $\varphi_p:A_{dR}\to\Q_p$ is injective for every $p$.

%
%
\section{Valuations of \uni{} periods}
\label{secval}
Let $X$ be a smooth projective variety over $\Z_p$. It is known that matrix coefficients for $F_{p}$ coming from the $n$-th level of the Hodge filtration on $H_{dR}(X)$ are divisible by $p^n$ if $\dim(X)<p$ (see \cite{Maz72}, p.\ 666). We formulate a version of this statement for \uni{} periods. Recall that there is a decreasing filtration $\fil^\bullet$ on $\A$ given by
\[
\fil^n\A:=\big\{(a_p)_p:\liminf_p v_p(a_p)\geq n\big\}.
\]
We also denote by $\fil^\bullet$ the Hodge filtration on $H_{dR}(X)$.

\begin{proposition}
\label{propval}
Suppose $X$ is smooth and projective over $\Q$. Then for all integers $n\geq 0$, the \uni{} Frobenius $F_{\A,X}$ takes $\fil^n H_{dR}(X)$ into $H_{dR}(X)\otimes\fil^n\A$.
\end{proposition}
\begin{proof}
Choose an integral model $\tilde{X}$ of $X$ over $\Z[1/N]$ for some positive integer $N$. For $p\nmid N$, let $X_p$ be the reduction of $\tilde{X}$ modulo $p$. There is a canonical isomorphism
\begin{equation}
\label{eqdrc}
H_{dR}(X)\otimes_\Q\Q_p\cong H_{crys}(X_p;\Z_p)\otimes_{\Z_p}\Q_p.
\end{equation}
Choose a basis $B$ for $H_{dR}(X)$ with the property that $B\cap\fil^nH_{dR}(X)$ is a basis for $\fil^n H_{dR}(X)$ for all $n$. There is an integer $M>N$ such that for all $p>M$, \eqref{eqdrc} takes $B$ to a $\Z_p$-basis of $H_{crys}(X_p;\Z_p)$. We choose $M$ large enough that $\fil^M H_{dR}(X)=0$. For $p>M$, it is known that the crystalline Frobenius on $H_{crys}(X_p;\Z_p)$ takes the $n$-th level of the Hodge filtration into $p^nH_{crys}(X_p;\Z_p)$. So for $p>M$, the matrix of $F_{p,X}$ with respect to our chosen basis has the property that columns corresponding to elements in $\fil^n H_{dR}(X)$ have each entry in $p^n\Z_p$. It follows that the corresponding columns in the matrix for $F_{\A,X}$ are in $\fil^n\A$, so that $F_{\A}$ takes $\fil^n H_{dR}(X)$ into $H_{dR}(X)\otimes\fil^n\A$.
\end{proof}

The strong form of our period conjecture is a kind of converse to Proposition \ref{propval}. Roughly speaking, the conjecture asserts that if a \uni{} periods is divisible by $p^n$ for all sufficiently large $p$, then that period is motivically equivalent something coming from the $n$-th level of the Hodge filtration.

To formulate the conjecture precisely, we fix a category $\cM$ of motives. The Hodge filtration on $\omega_{dR}$ induces a filtration on $\Pd$. If $F_{p,M}$ takes $\fil^n M_{dR}$ into $M_{dR}\otimes\fil^n\A$ for all $M\in\cM$ and $n\in\Z$ (Proposition \ref{propval} suggests we should expect this to be the case), then $\per_\A$ takes $\fil^n\Pd(M)$ into $\fil^n\A$. We expect that $\per_\A$ is compatible with the filtrations in a stronger sense. The strong form of our period conjecture is the following assertion about $\cM$.

\begin{repconj}{conups}[Strong \uni{} period conjecture]
For all integers $n$, we have
\[
\per_{\A}^{-1}\big(\fil^n\A\big)=\fil^n\Pd.
\]

\end{repconj}
Conjecture \ref{conups} is equivalent to the statement that $\per_{\A}$ induces an isomorphism of filtered algebras $\Pd\cong \cP_{\A}(\cM)$. Conjecture \ref{conups} implies Conjecture \ref{conup} because the filtration on $\Pd$ is separated. In some cases, it is possible to show directly that $\per_\A(\fil^n\Pd)\subset\fil^n\A$. This is the case for mixed Tate motives over $\Z$ (see Sec.\ \ref{secMTZ}). The hard part of the conjecture is the converse: if $x\in\Pd$ satisfies $\per(x)\in\fil^n\A$, then $x\in\fil^n\Pd$.

%
%
\section{Consequences of the period conjectures}
\label{secpcc}
In this section, we describe several consequences of Conjectures \ref{conup} and \ref{conups} applied to various categories of motives.

\subsection{Wieferich primes} 
\label{sswie}
Fix a rational number $r\neq0,\pm1$. Fermat's Little Theorem asserts that, for all primes $p$ not dividing the numerator or denominator of $r$,
\begin{equation}
\label{eqwie}
r^{p-1}\equiv 1\mod p.
\end{equation}
A \emph{Wieferich prime to the base $r$} is a prime which \eqref{eqwie} holds modulo $p^2$. A heuristic argument suggests that, for each $r$, the set of Wieferich primes is infinite but has density $0$. Little is known about the Wieferich primes, and it is an open problem to produce an $r$ for which there are infinitely many non-Wieferich primes. It is known \cite{Sil88} that the truth of the ABC-conjecture would imply there are infinitely many non-Wieferich primes to the base $2$.

\begin{theorem}
\label{thnwi}
If Conjecture \ref{conups} holds for the category of mixed Tate motives over $\Q$, then for every rational $r\neq 0,\pm1$, there exist infinitely many non-Wieferich primes to the base $r$.
\end{theorem}
\begin{proof}
For $r\in\Q\backslash\{0,1\}$, there is a Kummer motive $K(r)\in \MT(\Q)$ sitting in a short exact sequence
\[
0\to\Q(0)\to K(r)\to\Q(-1)\to 0.
\]
It follow from \cite{Del89}, \S2.9 that with respect to an appropriate basis, the matrix for $F_p$ on $K(r)_{dR}\otimes\Q_p$ is given by
\[
\ltwomat
1&\log_p(r^{p-1})\\
0&p
\rtwomat,
\]
where $\log_p$ is the $p$-adic logarithm. It is not hard to check that $\log_p(r^{p-1})$ is a non-zero element of $p\Z_p$ for all primes $p$ at which $r$ is a unit, and for these $p$, $v_p(\log_p(r^{p-1}))\geq 2$ precisely when $p$ is a Wieferich prime to the base $r$. Since $\fil^2 K(r)_{dR}=0$, Conjecture \ref{conups} implies that every non-zero \uni{} period of $K(r)$ has valuation exactly $1$ for infinitely many $p$. It follows that Conjecture \ref{conups} implies that there are infinitely many non-Wieferich primes to the base $r$.
\end{proof}

\subsubsection{Fermat's Last Theorem} In 1909, Wieferich proved that if the first case of Fermat's Last Theorem fails for a prime $p$, then $p$ must be a Wieferich prime to the base $2$. So Conjecture \ref{conups} would imply that the first case of Fermat's Last Theorem is true for infinitely many $p$. By comparison, the first proof that the first case of Fermat's Last Theorem is true for infinitely many $p$ appeared in 1985 \cite{Adl85}, so this gives a lower bound on the difficulty of proving Conjecture \ref{conups} for $\MT(\Q)$.

\subsection{Bernoulli numbers}
\label{sswo}

The Bernoulli numbers $B_n\in\Q$ are a sequence of rational numbers defined by
\[
\frac{x}{e^x-1}=\sum_{n\geq 0}B_n\frac{x^n}{n!}.
\]
While the denominator of $B_n$ is the product of those primes $p$ for which $p-1|n$, the numerators are much more mysterious. For $p$ prime, the Herbrand-Ribet theorem gives a connection between the set of Bernoulli numbers with numerator divisible by $p$ and the class group of the cyclotomic field $\Q(\zeta_p)$ (see \cite{Her32} and \cite{Rib76}).

Conjecture \ref{conups} applied to the category of mixed Tate motives over $\Z$ has the following consequence for the Bernoulli numbers.

\begin{theorem}
\label{thnwo}
If Conjecture \ref{conups} holds for the category of mixed Tate motives over $\Z$, then for every odd $k\geq 3$, there exist infinitely primes $p$ for which $p\nmid B_{p-k}$.
\end{theorem}
\begin{proof}
Let $\Pd$ be the ring of de Rham periods of $\MT(\Z)$, which is described in Sec.\ \ref{ssdr}. There is an element $\Zd(k)\in\fil^k\Pd\backslash\fil^{k+1}\Pd$, with the property that
\[
\per_p(\Zd(k))=\zeta_p(k)\equiv p^k\frac{B_{p-k}}{k}\mod p^{k+1}.
\]
Conjecture \ref{conups} for $\MT(\Z)$ then implies $\big(p^kB_{p-k}/k\big)\not\in\fil^{k+1}\A$, which is the statement that there are infinitely many $p$ not dividing $B_{p-k}$.
\end{proof}

\noindent Theorem \ref{thnwo} can be generalized considerably.
\begin{theorem}
\label{thgen}
Assume the truth of Conjecture \ref{conups} for $\MT(\Z)$, and let $f\in\Q[x_1,\ldots,x_n]$ be a non-zero polynomial. Then there are infinitely many primes $p$ for which $p$ does not divide the numerator of
\[
f\big(B_{p-3},B_{p-5},\ldots,B_{p-2n-1}\big).
\]
\end{theorem}
\noindent The proof is given in Sec.\ \ref{ssfmzv}.

\subsection{Fullness of the Ogus realization}

The following definition is a \uni{} analogue of a definition in \cite{Maz72}.
\begin{definition}
A \emph{\uni{} span} is a triple $(V,W,F)$, where $V$ and $W$ are finite-dimensional vector spaces over $\Q$ and $F:V\otimes\A\iso W\otimes\A$ is an $\A$-linear isomorphism. The collection of all \uni{} spans forms a neutral Tannakian category, which is denoted $\spa$
\end{definition}
There is a functor $\og\to\spa$, $(V,F)\mapsto (V,V,F)$ that ``forgets'' that the two vector spaces are the same. 

\begin{theorem}
\label{thfull}
Conjecture \ref{conup} holds for $\cM$ if and only if the functor $\cM\to\spa$, $M\mapsto (M_{dR},M_{dR},F_{\A,M})$ is full.
\end{theorem}
\begin{proof} 
The Betti-de Rham category over $\Q$ is the category whose objects are triples $(V,W,c)$, where $V$ and $W$ are finite-dimensional vector spaces over $\Q$ and
\[
c:V\otimes\C\iso W\otimes\C.
\]
It is well-known (see e.g.\ \cite{Hub17}, Proposition 13.2.8) that Conjecture \ref{conper} is equivalent to the assertion that the functor from $\cM$ to the Betti-de Rham category is full. The same proof works here.
\end{proof}

\begin{corollary}
\label{thog}
Conjecture \ref{conup} implies the Ogus realization is full.
\end{corollary}
\begin{proof}
Fullness of $\cM\to\spa$ implies fullness of $\cM\to\og$ because $\og\to\spa$ is faithful.
\end{proof}

\begin{remark}
Conjecture \ref{conup} is a priori stronger than the fullness of the Ogus realization. Conjecture \ref{conup} is equivalent to the statement that the crystalline Frobenius elements of $G_{dR}$ are Zariski dense, whereas the fullness of the Ogus realization is equivalent to the statement that the crystalline Frobenius elements generate a Zariski dense subgroup of $G_{dR}$.
\end{remark}

\subsection{Existence of algebraic cycles}

Here we describe a consequence of Conjecture \ref{conup} for algebraic cycles. We will need to assume the standard conjectures on algebraic cycles. Specifically, we assume that  for every smooth projective variety over $\Q$,
\begin{enumerate}
\item The K\"unneth projectors are algebraic.
\item Numerical equivalence of cycles equals homological equivalence.
\end{enumerate}
It is shown in \cite{Jan92} (Corollary 2) that condition (1) above implies that the category of pure numerical motives over $\Q$ (with coefficients in $\Q$) is neutral Tannakian. Condition (2) implies that $\omega_{dR}$ and $\omega_B$ are fibre functors, and the de Rham-Betti comparison and crystalline Frobenius map come from the corresponding operations on cohomology.
\begin{theorem}
\label{thcyc}
Assume that the assertions (1) and (2) above hold, and let $\cM$ be the category of pure numerical motives over $\Q$. The truth of Conjecture  \ref{conup} for $\cM$ is equivalent to the assertion that, whenever $X$ is a smooth projective variety, $n\in\Z_{\geq 0}$, and $\alpha$, $\beta\in H_{dR}^{2n}(X)(n)$ satisfy $F_{p,X}(\alpha)= \beta$ for all but finitely many $p$, then $\alpha$ is algebraic (so in fact $\alpha=\beta$).
\end{theorem}
Here $H_{dR}(X)(n)$ is a Tate twist, meaning $F_{p,X}$ is scaled by $p^{-n}$.

\begin{proof}
By Theorem \ref{thfull}, Conjecture \ref{conup} for $\cM$ is equivalent to the assertion that $\cM\to\spa$ is full. For $M\in\cM$, write $Sp(M)$ for the \uni{} span associated with $M$. Then $\cM\to\spa$ is full if and only if, for every $M\in\cM$,
\[
\hom_{\cM}(\Q(0),M)=\hom_{\spa}(Sp(\Q(0)),Sp(M)).
\]
It suffices to consider only those $M$ of the form $H(X)(n)$. Morphisms of motives $\Q(0)\to H(X)(n)$ are in natural bijection with algebraic classes in $H^{2n}_{dR}(X)$. Morphisms of \uni{} spans $Sp(\Q(0))\to Sp(H(X)(n))$ correspond to pairs $\alpha,\beta\in H_{dR}(X)(n)$ such that $F_\A(\alpha)=\beta$, and such $\alpha$ and $\beta$ must live in $H_{dR}^{2n}(X)(n)$ for weight reasons. This completes the proof.
\end{proof}

\begin{remark}
An \emph{Ogus cycle} is an element $\alpha\in H_{dR}^{2n}(X)(n)$ such that $F_{p,X}(\alpha)=\alpha$ for all sufficiently large $p$. It is conjectured that every Ogus cycles is algebraic (see \cite{Ogu82}, \S4), and statement is equivalent to the fullness of the Ogus realization.
\end{remark}

\subsection{Galois theory}
\label{ssgal}
Let $\cM$ be a category of motives. As the coordinate ring of an affine group scheme, $\Pd$ has the structure of a commutative Hopf algebra over $\Q$. The ring of motivic periods $\Pm$ is an algebra-comodule over $\Pd$, and the corresponding group action
\[
\sp\,\Pd\curvearrowright\sp\,\Pm
\]
makes $\sp\,\Pm$ into a torsor for $\sp\,\Pd$.

If Conjecture \ref{conup} holds for $\cM$, the Hopf algebra structure descends to $\cP_\A(\cM)$. If in addition Conjecture \ref{conper} holds for $\cM$, then the coaction also descends. In this case we get ring homomorphisms
\begin{align}
\label{eqcomul}\Delta_{\A}:\cP_\A(\cM)&\to\cP_\A(\cM)\otimes \cP_\A(\cM),\\
\label{eqcoact}\Delta_\C:\cP_\C(\cM)&\to\cP_\A(\cM)\otimes \cP_\C(\cM),
\end{align}
making $\cP_\A(\cM)$ into a Hopf algebra and making $\cP_\C(\cM)$ into an algebra-comodule for $\cP_\A(\cM)$.

In practice, \eqref{eqcomul} and \eqref{eqcoact} can be computed explicitly. The computation depends a priori on some choices, and Conjecture \ref{conper} and \ref{conup} imply that the result is independent of the choices. Concretely, suppose $M\in\cM$. Given $\omega\in M_{dR}$ and $\eta\in M_{dR}^\du$ (where $M_{dR}^\du$ is the $\Q$-linear dual), we get a \uni{} period of $M$
\[
\big\langle\omega,\eta\big\rangle_\A:=
\bigg(\big\langle F_{p,M}\omega,\eta\big\rangle\bigg)_p\in\A,
\]
and $\cP_\A(\cM)$ is spanned over $\Q$ by elements of this form. To compute the comultiplication \eqref{eqcomul}, we choose a basis $\{v\}$ of $M_{dR}$, with dual basis $\{v^\du\}$, and we have
\begin{equation}
\label{eqcomex}
\Delta_\A\bigg(\big\langle\omega,\eta\big\rangle_\A\bigg)=\sum_{v}\big\langle\omega,v^\du\big\rangle_\A\otimes \big\langle v,\eta\big\rangle_\A.
\end{equation}
Similarly, for $\omega\in M_{dR}$ and $\gamma\in M_B^\du$, we get a complex period of $M$
\[
\big\langle\omega,\gamma\big\rangle_\C:=\big\langle comp(\omega),\eta\big\rangle\in\C.
\]
The coaction \eqref{eqcoact} is given by
\begin{equation}
\label{eqcoaex}
\Delta_\C\bigg(\big\langle\omega,\gamma\big\rangle_\C\bigg)=\sum_{v}\big\langle\omega,v^\du\big\rangle_\A\otimes \big\langle v,\gamma\big\rangle_\C.
\end{equation}

\begin{remark}
The formulas \eqref{eqcomex} and \eqref{eqcoaex} for the comultiplication and coaction make sense for varieties, even without an underlying category of motives.
\end{remark}

%
%
\section{$\cA$-valued periods}
\label{appp}

Kaneko and Zagier (unpublished) observe that certain elements of the ring
\[
\cA:=\frac{\prod_p\Z/p\Z}{\bigoplus_p \Z/p\Z}\cong\frac{\fil^0\A}{\fil^1\A}
\]
are analogous to periods. In this section we construct a version of \uni{} periods living in $\cA$. Applied to the category of mixed Tate motives over $\Z$, our construction recovers the finite multiple zeta values of Kaneko-Zagier (see Sec.\ \ref{ssfmzv}).

We first give a concrete definition for varieties. Suppose $X$ is a smooth projective variety over $\Q$. By Proposition \ref{propval},  $\p^{-n}F_{\A,X}$ takes $\fil^n H_{dR}(X)$ into $H_{dR}(X)\otimes\fil^0\A$, and we write
\[
F_{\cA,X}^{(n)}:\fil^n H_{dR}(X)\to H_{dR}(X)\otimes\cA
\]
for the composition of $\p^{-n}F_{\A,X}:\fil^n H_{dR}\to H_{dR}(X)\otimes\fil^0\A$ with the projection $\fil^0\A\to\cA$. If we choose a $\Q$-basis for $\fil^n H_{dR}(X)$, and extend it to a basis of $H_{dR}(X)$, we can represent $F_{\cA,X}^{(n)}$ by a non-square matrix with entries in $\cA$.

\begin{definition}
An \emph{$\cA$-valued period of $X$} is a $\Q$-linear combination of matrix coefficients for $F_{\cA,X}^{(n)}$ (as $n$ varies). We denote the set of $\cA$-valued periods of $X$ by $\cP_\cA(X)\subset\cA$.
\end{definition}

Concretely, the matrix for the \uni{} Frobenius on $H_{dR}(X)$ has columns that are divisible by various powers of $p$, according to the Hodge filtration on $H_{dR}(X)$, and we get an $\cA$-valued period by taking a matrix entry whose column is in the $n$-th level of the Hodge filtration, dividing by $p^n$, and reducing modulo $p$.

To understand $\cA$-valued periods in terms of motivic periods, we need a category $\cM$ of motives with Ogus realization that is known to satisfy the easier direction of Conjecture \ref{conups}. We make the following definition, which is the \uni{} version of a definition in \cite{Maz72}, p. 665.
\begin{definition}
\label{defdiv}
Let $\cM$ be a category of motives. We say the Ogus realization on $\cM$ is \emph{divisible} if $\per_\A(\fil^n\Pd)\subset\fil^n\A$ for all $n$.
\end{definition}
For example, it is known that the category mixed Tate motives over $\Z$ has divisible Ogus realization (see Sec.\ \ref{secMTZ}). The truth of the standard conjectures on algebraic cycles would imply that the Ogus realization for pure numerical motives is divisible.

For the remainder of this section, we assume $\cM$ is a category of motives equipped with a divisible Ogus realization. Consider 
\[
\Pd_0:=\frac{\fil^0\Pd}{\fil^1\Pd},
\]
which is the degree $0$ part of the associated graded ring of $\Pd$. The following definition is new.
\begin{definition}
The \emph{$\cA$-valued period map} is the ring homomorphism
\[
\per_{\cA}:\Pd_0\to\cA
\]
induced by $\per_{\A}$. We write $\cP_{\cA}$ for the image of $per_{\cA}$, and we call $\cP_{\cA}$ the \emph{ring of $\cA$-valued periods} of $\cM$.
\end{definition}

We expect that the map $\per_\cA$ is injective (this statement could be called the \emph{$\cA$-valued period conjecture}). Injectivity of $\per_\cA$ is essentially equivalent to Conjecture \ref{conups}.

\begin{theorem}
\label{thfups}
Conjecture \ref{conups} implies that $\per_\cA$ is injective. The converse is true if $\cM$ admits Tate twists.
\end{theorem}
\begin{proof}
It is immediate that Conjecture \ref{conups} implies $\per_\A$ is injective. Conversely, suppose that $\per_\A$ is injective, and suppose $\alpha\in\Pd$ satisfies $\per_\A(\alpha)\in\fil^n\A$. Let $m=\max\{M\leq n:\alpha\in\fil^M\Pd\}$. We will show that $m=n$.

For the sake of contradiction, suppose $m<n$. By assumption there is a rank-$1$ object $\Q(-1)\in\cM$ (the Lefschetz motive), whose de Rham period $\Ld\in\Pd$ is mapped to $\p$ by $\per_\A$. We have $(\Ld)^{-m}\alpha\in\fil^0\Pd$, and 
\[
\per_\A\big((\Ld)^{-m}\alpha\big)=\p^{-m}\per_\A(\alpha)\in\fil^{n-m}\A\subset\fil^1\A.
\]
This shows $\per_\cA((\Ld)^{-m}\alpha)=0$. Because $\per_\cA$ is injective by assumption, we conclude $(\Ld)^{-m}\alpha\in\fil^1\Pd$. It follows that $\alpha\in\fil^{m+1}\Pd$, contradicting the maximality of $m$.
\end{proof}

The finite multiple zeta values of Kaneko-Zagier are the $\cA$-valued periods of the category of mixed Tate motives over $\Z$ (Theorem \ref{thfm} below). The category $\MM$ of Artin motives over $\Q$ has trivial Hodge filtration, and for this category the $\cA$-valued period map is a map $\per_\A:\Pd\to\cA$. The paper \cite{Ros18b} involves an application of the $\cA$-valued period map for $\MM$ to an analogue of the Skolem-Mahler-Lech theorem.

\begin{remark}
If one is interested in congruences modulo $p^n$, one can consider the ring
\[
\cA_n=\frac{\prod_p\Z/p^n\Z}{\bigoplus_p \Z/p^n\Z},
\]
and there is an $\cA_n$-valued period map
\[
\per_{\cA_n}:\frac{\fil^0\Pd}{\fil^n\Pd}\to\cA_n.
\]
If $\cM$ admits Tate twists, then Conjecture \ref{conups} is equivalent to the statement that $\per_{\cA_n}$ is injective for one (equivalently, every) positive integer $n$.
\end{remark}

%
%
\section{$\Ai$-valued periods}
\label{appc}
Some arithmetically interesting quantities can be expressed as infinite sums of $p$-adic periods in a manner that is uniform in $p$  (we give some examples in Sec.\ \ref{mtzc}). For example, for $k\geq 2$ an integer, the $p$-adic zeta value $\zeta_p(k)$ is a $p$-adic period, and a result of Washington \cite{Was98} expresses a harmonic number in terms of $\zeta_p(k)$:
\begin{equation}
\label{eqwash}
p^s\sum_{\substack{n=1\\p\nmid n}}^{pr}\frac{1}{n^s}=\sum_{k=0}^\infty (-1)^k {r+k\choose k+1} r^{k+1}\zeta_p(s+k+1),
\end{equation}
This formula has a generalization due to Jarossay \cite{Jar15a}.

Here we describe a version of \uni{} periods for treating these infinite sums. This notion of period was considered in \cite{Ros18} for mixed Tate motives over $\Z$.

The filtration $\fil^\bullet$ on $\A$ is neither exhaustive nor separated. To deal with infinite sums of \uni{} periods, we replace $\A$ with
\[
\Ai:=\frac{\bigcup_n \fil^n\A}{\bigcap_n\fil^n\A}=\frac{\left\{(a_p)\in\prod_p\Q_p:v_p(a_p)\text{ bounded below}\right\}}{\left\{(a_p)\in\prod_p\Q_p:v_p(a_p)\to\infty \text{ as }p\to\infty\right\}}.
\]
We write $\fil^\bullet$ for the induced filtration on $\Ai$, which is exhaustive and separated. The ring $\Ai$ is complete with respect to the topology induced by $\fil^\bullet$.

Fix a category of motives $\cM$ with divisible Ogus realization (Definition \ref{defdiv}). This means $per_{\A}$ takes $\fil^n\Pd$ into $\fil^n\A$ for each $n$, so $\per_\A$ induces a continuous homomorphism of the completions.
\begin{definition}
The \emph{completed \uni{} period map} is the continuous ring homomorphism
\[
\hper:\Pdh\to\Ai,
\]
where $\Pdh$ is the completion of $\Pd$ with respect to the Hodge filtration. We write $\cP_{\Ai}(\cM)$ for the image of $\hper$, and we call $\cP_{\Ai}(\cM)$ the \emph{ring of $\Ai$-valued periods} of $\cM$.
\end{definition}

In terms of the example \eqref{eqwash}, we will see in Sec.\ \ref{secMTZ} that $\zeta_p(k)$ is a $p$-adic period of the category of mixed Tate motives over $\Z$. For this category, there are elements $\Zd(k)\in\Pd$ that $\per_\A$ maps to $\big(\zeta_p(k)\big)_p$, and \eqref{eqwash} implies $\per_\A$ takes the element
\[
\sum_{k=0}^\infty (-1)^k {r+k\choose k+1} r^{k+1}\Zd(s+k+1)\in\Pdh
\]
to
\[
\lp p^s\sum_{\substack{n=1\\p\nmid n}}^{pr}\frac{1}{n^s}\rp_p\in\A.
\]

\begin{remark}
Conjecture \ref{conups}  for $\cM$ is equivalent to the statement $\hper$ induces an isomorphism of filtered algebras $\Pdh\cong \cP_{\Ai}(\cM)$. In particular, Conjecture \ref{conups} would imply $\cP_{\Ai}(\cM)$ is complete with respect to $\fil^\bullet$.
\end{remark}

%
%
\section{Mixed Tate motives over $\Z$}
\label{secMTZ}
The theory of mixed Tate motives and their periods is well-developed. In this section, we compute the \uni{} periods of $\MT(\Z)$. We also compute the $\cA$-valued and $\Ai$-valued periods of $\MT(\Z)$.

Mixed Tate motives are described in \cite{Del89}. In \cite{Del05} a construction is given for the category of mixed Tate motives unramified over the ring of $S$-integers in a number field. Theory of motivic iterated integrals appears in \cite{Gon05}.  A computation of the periods of $\MT(\Z)$ is given in \cite{Bro12}. The crystalline realization (for each prime $p$) on $\MT(\Z)$ is constructed in \cite{Yam10}, and these assemble to give an Ogus realization.

Recall that a \emph{composition} is a finite ordered list $\bs=(s_1,\ldots,s_k)$ of positive integers. The \emph{weight} of $\bs$ is $|\bs|:=s_1+\ldots+s_k$. The various types of periods of $\MT(\Z)$ are indexed by compositions.

\subsection{Motivic and complex periods}

For $\bs=(s_1,\ldots,s_k)$ a composition satisfying $s_1\geq 2$, the \emph{multiple zeta value} (or \emph{MZV}) is
\begin{equation}
\label{eqdefmzv}
\zeta(\bs):=\sum_{n_1>\ldots>n_k\geq 1}\frac{1}{n_1^{s_1}\ldots n_k^{s_k}}\in\R.
\end{equation}
Using the so-called shuffle regularization, it is possible to extend the definition to the case $s_1=1$. It is known the the MZVs are periods of mixed Tate motives over $\Z$.

The ring $\cP^\fm$ of motivic periods of $\MT(\Z)$ contains elements $\zeta^\fm(\bs)$, the \emph{motivic MZVs}, as $\bs$ ranges over the compositions. The period map takes $\zm(\bs)$ to $\zeta(\bs)$. The elements $\zeta^\fm(\bs)$ satisfy many algebraic relations, and a relation satisfied by $\zeta(\bs)$ is called \emph{motivic} if the corresponding relation holds for $\zm(\bs)$. There is also an invertible element $\L\in\cP^\fm$, the \emph{motivic Lefschetz period}, which maps to $2\pi i$ under the period map. Euler's calculation that $\zeta(2)=\pi^2/6$ lifts to the motivic setting, and we have $\zeta^\fm(2)=-(\Lm)^2/24$.

\begin{proposition}[Brown \cite{Bro12}]
The ring of motivic periods of $\MT(\Z)$ is spanned as a $\Q[\Lm,(\Lm)^{-1}]$-module by the motivic MZVs.
\end{proposition}
\begin{corollary}
The ring of complex periods of $\MT(\Z)$ is spanned as a $\Q[2\pi i,(2\pi i)^{-1}]$-module by the MZVs.
\end{corollary}

There is a $\Z$-grading on $\cP^\fm$ called the grading by weight,\footnote{The weight here is one half of the motivic weight} where $\L$ has weight $1$ and $\zm(\bs)$ has weight $\bs$. One often works with the subalgebra $\cH\subset\cP^\fm$ spanned by the motivic MZVs. The ring $\cH$ is called the \emph{ring of motivic MZVs}, and it has the advantage of being $\mathbb{N}$-graded with finite-dimensional graded pieces.

\subsection{De Rham and $p$-adic periods}
\label{ssdr}
For $\bs$ a composition, there are $p$-adic analogues $\zeta_p(\bs)\in\Q_p$ of the MZVs,\footnote{The $p$-adic MZVs we use here were constructed by Deligne, but there is a different version due to Furusho \cite{Fur04}.} arising from the action of Frobenius on the crystalline fundamental group of the thrice punctured line (see \cite{Del02}, or \S5.28 of \cite{Del05}). The $p$-adic MZVs are $p$-adic periods of $\MT(\Z)$. Results of Jarossay \cite{Jar15} give explicit computations of $\zeta_p(\bs)$.

The ring $\Pd$ of de Rham periods of $\MT(\Z)$ contains elements $\Zd(\bs)$, the \emph{de Rham MZVs}, as $\bs$ ranges over the compositions. Objects of $\MT(\Z)$ are unramified at every prime, so for each $p$ there is a $p$-adic period map $\per_p:\Pd\to\Q_p$, which takes $\Zd(\bs)$ to $\zeta_p(\bs)$. There is an invertible element $\Ld$, the de Rham Lefschetz period, which the period map takes to $p$.

\begin{proposition}[Brown \cite{Bro14}]
\label{pdr}
The ring of de Rham periods of $\MT(\Z)$ is spanned as a $\Q[\Ld,(\Ld)^{-1}]$-module by the de Rham MZVs.
\end{proposition}
\begin{corollary}
The ring of $p$-adic periods of $\MT(\Z)$ is spanned over $\Q$ by the $p$-adic MZVs.
\end{corollary}
The de Rham MZVs satisfy the same $\Q$-linear relations as the motivic MZVs, along with the additional relation $\Zd(2)=0$, thus the $p$-adic MZVs also satisfy these relations.

\begin{remark}
The $p$-adic period map is not injective because it kills $\Ld-p$. It is conjectured that for each $p$, the kernel of $\per_p$ is generated as an ideal by $\Ld-p$.
\end{remark}

\subsection{\Uni{} periods}
The \uni{} period map $\per_\A$ takes $\Zd(\bs)$ to the \emph{\uni{} MZV}
\[
\zeta_\p(\bs):=\big(\zeta_p(\bs)\big)_p\in\A,
\]
and takes $\Ld$ to $\p$.
\begin{proposition}
The ring of \uni{} periods of $\MT(\Z)$ is spanned as a $\Q[\p,\p^{-1}]$-module by the \uni{} MZVs.
\end{proposition}
\begin{proof}
Follows from Proposition \ref{pdr} by applying $\per_\A$.
\end{proof}

Like $\cP^\fm$, the ring $\Pd$ is graded by weight, where $\Ld$ has weight $1$ and $\Zd(\bs)$ has weight $|\bs|$. The Hodge filtration is induced by the grading: $\fil^n\Pd$ is spanned by elements of weight $n$ and higher. It is known \cite{Cha17} that the $p$-adic multiple zeta values satisfy
\begin{equation}
\label{eqzpc}
\zeta_p(\bs)\in p^{|\bs|}\Z_p
\end{equation}
for all $p>|\bs|$, which means that the Ogus realization on $\MT(\Z)$ is divisible, i.e.\ $\per_\A(\fil^n\Pd)\subset\fil^n\A$ for all $n$. So $\cA$-valued periods (Sec.\ \ref{appp}) and $\Ai$-valued periods (Sec.\ \ref{appc}) are defined for $\MT(\Z)$.

The statement that $\hper$ is injective for $\MT(\Z)$, which is equivalent to Conjecture \ref{conups} for $\MT(\Z)$, was stated by the author in \cite{Ros18} (Conjecture 1.3). A related conjecture can be found in \cite{Jar16d} (Conjecture 7.7).

\subsection{Multiple harmonic sums}
Multiple harmonic sums are truncated versions of the MZVs. As we will see, they are related to the $\cA$-valued periods and $\Ai$-valued periods of $\MT(\Z)$.
\begin{definition}
Let $\bs=(s_1,\ldots,s_k)$ be a composition and $N$ a positive integer. The quantity
\[
H_N(s_1,\ldots,s_k):=\sum_{N\geq n_1>\ldots>n_k\geq 1}\frac{1}{n_1^{s_1}\ldots n_k^{s_k}}\in\Q
\]\
is called a \emph{multiple harmonic sum}.
\end{definition}

Multiple harmonic sums are known to have interesting arithmetic properties, particularly in the case $N=p-1$ with $p$ prime (see \cite{Hof04a}, \cite{Zha08}).
The following formula of Jarossay \cite{Jar15a} expresses multiple harmonic sums $\H(\bs)$ in terms of $p$-adic MZVs.
\begin{theorem}
\label{pj}
Let $\bs=(s_1,\ldots,s_k)$ be a composition. There is a $p$-adically convergent series identity
\begin{gather}
\label{eqjar}
\hspace{-20mm}p^{|\bs|}\H(\bs)=\sum_{i=0}^k \sum_{\ell_1,\ldots,\ell_i\geq 0}(-1)^{s_1+\ldots+s_i} \prod_{j=1}^i {s_j+\ell_j-1\choose\ell_j}
\hfill\\
\hspace{30mm}\hfill \zeta_p(s_i+\ell_i,\ldots,s_1+\ell_1)\zeta_p(s_{i+1},\ldots,s_k).
\end{gather}
\end{theorem}
It follows from \eqref{eqzpc} that the convergence of the infinite series on the right hand side of \eqref{eqjar} is uniform in $p$, in the sense that it induces a convergent series identity in $\Ai$.

\subsection{$\cA$-valued periods}
\label{ssfmzv}
To study congruences between multiple harmonic sums $\H$ modulo $p$, Kaneko and Zagier made the following definition.
\begin{definition}[Kaneko-Zagier, unpublished]
For $\bs$ a composition, the \emph{finite multiple zeta value} is defined to be
\[
\zeta_{\cA}(\bs):=\big(\H(\bs)\mod p\big)_p\in\cA.
\]
\end{definition}
An equality between finite multiple zeta values corresponds to a congruence modulo $p$ that holds for all but finitely many $p$.

\begin{theorem}
\label{thfm}
The ring of $\cA$-valued periods of $\MT(\Z)$ is spanned as a vector space over $\Q$ by the finite MZVs.
\end{theorem}
\begin{proof}
The degree $0$ graded piece of $\Pd$ is the $\Q$-span of the elements
\begin{equation}
\label{eqnm}
(\Ld)^{-|\bs|}\Zd(\bs),
\end{equation}
as $\bs$ ranges through the compositions. The formula \eqref{eqjar} implies that the image of \eqref{eqnm} under the map $\per_\A$ is in the $\Q$-span of the finite MZVs. A result of Yasuda \cite{Yas16} implies that conversely, we can solve for the finite MZVs as a $\Q$-linear combination of the images of \eqref{eqnm} under $\per_\A$. 
\end{proof}

Kaneko-Zagier conjectured that relations between the finite multiple zeta values are governed by relations among the real multiple zeta values. Specifically, the conjecture says that the $\Q$-linear relations satisfied by the finite multiple zeta values are precisely the same as the $\Q$-linear relations satisfied by the \emph{symmetrized multiple zeta values}
\[
\zeta^{s}(s_1,\ldots,s_k):=\sum_{i=0}^k (-1)^{s_1+\ldots+s_k}\zeta(s_i,\ldots,s_1)\zeta(s_{i+1},\ldots,s_k)
\]
modulo $\zeta(2)$. We show that Kaneko-Zagier's conjecture is essentially equivalent to Conjecture \ref{conups} for $\MT(\Z)$.

\begin{theorem}
\label{thfmzv}
The finite multiple zeta values satisfy every relation satisfied by the motivic version of the symmetrized multiple zeta values modulo $\zm(2)$. Conjecture \ref{conups} for $\MT(\Z)$ is equivalent to the assertion that the finite MZVs satisfy precisely the same relations as the motivic symmetrized MZVs modulo $\zm(2)$.
\end{theorem}
\begin{proof}
Yasuda's result \cite{Yas16} implies that $\Pd_0$ is the $\Q$-span of the the elements
\begin{equation}
\label{eqzds}
\bar{\zeta}^{\dr,s}(s_1,\ldots,s_k):=(\Ld)^{-|\bs|}\sum_{i=0}^k (-1)^{s_1+\ldots+s_k}\Zd(s_i,\ldots,s_1)\Zd(s_{i+1},\ldots,s_k).
\end{equation}
The elements \eqref{eqzds} satisfy precisely the same relations as the motivic symmetrized multiple zeta values modulo $\zm(2)$. Jarossay's formula \eqref{eqjar} implies that $\per_\cA$ takes $\bar{\zeta}^{\dr,s}(\bs)$ to $\zeta_\cA(\bs)$. This proves that the finite multiple zeta values satisfy every relations satisfied by the motivic symmetrized MZVs modulo $\zm(2)$. The converse is the statement that $\per_\cA:\Pd_0\to\cA$ is injective, which by Theorem \ref{thfups} is equivalent to Conjecture \ref{conups} for $\MT(\Z)$.
\end{proof}

We can now prove the modulo $p$ independence statement for Bernoulli numbers, which was stated in Sec.\ \ref{sswo}.
\begin{proof}[Proof of Theorem \ref{thgen}]
For $n\geq 3$ odd, the elements
\begin{equation}
\label{eqeq}
n(\Ld)^{-n}\Zd(n)\in\Pd_0
\end{equation}
are algebraically independent. If we assume Conjecture \ref{conups} for $\MT(\Z)$, then by Theorem \ref{thfups}, $\per_\A$ maps \eqref{eqeq} to algebraically independent elements of $\A$. The Theorem now follow from the well-known formula
\[
np^{-n}\zeta_p(n)\equiv B_{p-n}\mod p
\]
for $p$ sufficiently large.
\end{proof}

\subsection{$\Ai$-valued periods}
\label{mtzc}
Jarossay's formula \eqref{eqjar} expresses $\H(\bs)$ as a uniformly convergent series in terms of $p$-adic multiple zeta values, which implies that
\[
H_{\p-1}(\bs):=\big(\H(\bs)\big)\in\A
\]
is in the image of the completed \uni{} period map (it is the image of the element of $\Pdh$ obtained by replacing each $\zeta_p$ on the right hand side of \eqref{eqjar} with $\Zd$). Many other combinatorially defined quantities depending on $p$ can be expressed in terms of the $\H$, and can be shown to be $\Ai$-valued periods of $\MT(\Z)$. An explicit description of the ring of $\Ai$-valued periods of $\MT(\Z)$ was computed in \cite{Ros18}.
\begin{definition}[\cite{Ros18}, Theorem 3.3]
The \emph{MHS algebra} is the subalgebra of $\Ai$ consisting of those $\alpha\in\Ai$ for which there exist sequences $a_n\in\Q$, $b_n\in\Z$ with $b_n\to\infty$, and compositions $\bs_n$, such that
\begin{equation}
\label{eqconv}
\alpha=\sum_{n=0}^\infty a_n\, \p^{b_n}H_{\p-1}(\bs_n).
\end{equation}
\end{definition}
Observe that the condition $b_n\to\infty$ guarantees that \eqref{eqconv} converges.

\begin{theorem}[\cite{Ros18}, Theorem 3.3]
The ring of $\Ai$-valued periods of the category $\MT(\Z)$ is the MHS algebra.
\end{theorem}

As a consequence, elements of the MHS algebra can be lifted to $\Pdh$, and if we assume the truth of Conjecture \ref{conups} for $\MT(\Z)$, we get a Galois theory for the MHS algebra. Some aspects of the Galois action are computed in \cite{Ros18}.

Many combinatorially-defined sequences can be shown to be in the MHS algebra (several examples are given in \cite{Ros18}, \S7). For example, if $f(x)$, $g(x)\in\Z[x]$ have positive leading coefficient, then the sequence of binomial coefficients
\[
{f(\p)\choose g(\p)}:=\Bigg({f(p)\choose g(p)}\Bigg)_p \in\A
\]
is in the MHS algebra (\cite{Ros16a}, Theorem 7.11). In \cite{Ros18}, \S4, there is an algorithm (along with a link to an implementation) that takes as input a positive integer $n$ and two elements $(a_p)$, $(b_p)$ of the MHS algebra, and gives as ouput either a proof that $a_p\equiv b_p$ mod $p^n$ for $p$ large, or a proof that Conjecture \ref{conups} implies there are infinitely many $p$ for which $a_p\not\cong b_p$ mod $p^n$.


\bibliographystyle{alpha}
\bibliography{jrbiblio}

\newcommand{\noop}[1]{} \def\cprime{$'$}
\begin{thebibliography}{HMS17}

\bibitem[AHB85]{Adl85}
L.~M. Adleman and D.~R. Heath-Brown.
\newblock The first case of {F}ermat's last theorem.
\newblock {\em Invent. Math.}, 79(2):409--416, 1985.

\bibitem[And04]{And04}
Yves Andr{\'e}.
\newblock Une introduction aux motifs: motifs purs, motifs mixtes,
  p{\'e}riodes.
\newblock 2004.

\bibitem[Bro12]{Bro12}
Francis Brown.
\newblock Mixed {T}ate motives over {$\Bbb Z$}.
\newblock {\em Ann. of Math. (2)}, 175(2):949--976, 2012.

\bibitem[Bro14]{Bro14}
Francis Brown.
\newblock Single-valued motivic periods and multiple zeta values.
\newblock In {\em Forum of Mathematics, Sigma}, volume~2. Cambridge University
  Press, 2014.

\bibitem[Bro17]{Bro17}
Francis Brown.
\newblock Notes on motivic periods.
\newblock {\em Commun. Number Theory Phys.}, 11(3):557--655, 2017.

\bibitem[Cha17]{Cha17}
Andre Chatzistamatiou.
\newblock On integrality of $p$-adic iterated integrals.
\newblock {\em Journal of Algebra}, 474:240--270, 2017.

\bibitem[Del89]{Del89}
Pierre Deligne.
\newblock Le groupe fondamental de la droite projective moins trois points.
\newblock In {\em Galois groups over {${\bf Q}$} ({B}erkeley, {CA}, 1987)},
  volume~16 of {\em Math. Sci. Res. Inst. Publ.}, pages 79--297. Springer, New
  York, 1989.

\bibitem[Del02]{Del02}
Pierre Deligne.
\newblock Periods for the fundamental group.
\newblock {\em Arizona winter school 2002, course notes and project
  description}, 2002.

\bibitem[DG05]{Del05}
Pierre Deligne and Alexander~B. Goncharov.
\newblock Groupes fondamentaux motiviques de {T}ate mixte.
\newblock {\em Ann. Sci. \'Ecole Norm. Sup. (4)}, 38(1):1--56, 2005.

\bibitem[DM82]{Del82}
Pierre Deligne and James Milne.
\newblock Tannakian categories.
\newblock In {\em Hodge cycles, motives, and {S}himura varieties}, volume 900
  of {\em Lecture Notes in Mathematics}, pages 101--228. Springer-Verlag,
  Berlin-New York, 1982.

\bibitem[Fur04]{Fur04}
Hidekazu Furusho.
\newblock {$p$}-adic multiple zeta values. {I}. {$p$}-adic multiple
  polylogarithms and the {$p$}-adic {KZ} equation.
\newblock {\em Invent. Math.}, 155(2):253--286, 2004.

\bibitem[Gon05]{Gon05}
A.~B. Goncharov.
\newblock Galois symmetries of fundamental groupoids and noncommutative
  geometry.
\newblock {\em Duke Math. J.}, 128(2):209--284, 2005.

\bibitem[Her32]{Her32}
Jacques Herbrand.
\newblock Sur les classes des corps circulaires.
\newblock {\em Journal de Math{\'e}matiques Pures et Appliqu{\'e}es},
  11:417--441, 1932.

\bibitem[HMS17]{Hub17}
Annette Huber and Stefan M{\"u}ller-Stach.
\newblock {\em Periods and Nori Motives}, volume~65 of {\em Ergebnisse der
  Mathematik und ihrer Grenzgebiete}.
\newblock Springer Berlin Heidelberg, 2017.

\bibitem[Hof04]{Hof04a}
Michael~E. Hoffman.
\newblock Quasi-symmetric functions and mod $p$ multiple harmonic sums.
\newblock 2004.

\bibitem[Jan92]{Jan92}
Uwe Jannsen.
\newblock Motives, numerical equivalence, and semi-simplicity.
\newblock {\em Inventiones mathematicae}, 107(1):447--452, 1992.

\bibitem[Jar15a]{Jar15}
David Jarossay.
\newblock Direct computation of cyclotomic $p$-adic multiple zeta values.
\newblock 2015.

\bibitem[Jar15b]{Jar15a}
David Jarossay.
\newblock Une notion de multiz{\^e}tas finis associ{\'e}e au {F}robenius du
  groupe fondamental de $\mathbb{P}^1\backslash\{0,1,\infty\}$.
\newblock {\em Comptes Rendus Mathematique}, 353(10):877--882, 2015.

\bibitem[Jar16]{Jar16d}
David Jarossay.
\newblock An explicit theory of $\pi_{1}^\mathrm{un,crys}(\mathbb{P}^{1} -
  \{0,\mu_{N},\infty\})$ - {II}-3 : {S}equences of multiple harmonic sums
  viewed as periods, 2016.

\bibitem[Ked09]{Ked09}
Kiran~S Kedlaya.
\newblock p-adic cohomology.
\newblock {\em Algebraic Geometry, Seattle 2005: 2005 Summer Research
  Institute, July 25-August 12, 2005, University of Washington, Seattle,
  Washington}, 2:667, 2009.

\bibitem[Maz72]{Maz72}
Barry Mazur.
\newblock Frobenius and the {H}odge filtration.
\newblock {\em Bulletin of the American Mathematical Society}, 78(5):653--667,
  1972.

\bibitem[Ogu82]{Ogu82}
Arthur Ogus.
\newblock Hodge cycles and crystalline cohomology.
\newblock In {\em Hodge cycles, motives, and shimura varieties}, pages
  357--414. Springer, 1982.

\bibitem[Rib76]{Rib76}
Kenneth~A Ribet.
\newblock A modular construction of unramified $p$-extensions of
  $\mathbb{Q}(\mu_p)$.
\newblock {\em Inventiones mathematicae}, 34(3):151--162, 1976.

\bibitem[Ros16]{Ros16a}
Julian Rosen.
\newblock The {MHS} algebra and supercongruences.
\newblock 2016.

\bibitem[Ros18a]{Ros18}
Julian Rosen.
\newblock The completed finite period map and {G}alois theory of
  supercongruences.
\newblock {\em International Mathematics Research Notices}, page rny004, 2018.

\bibitem[Ros18b]{Ros18b}
Julian Rosen.
\newblock Recurrent sequences and de rham periods.
\newblock 2018.

\bibitem[Sil88]{Sil88}
Joseph~H Silverman.
\newblock Wieferich's criterion and the abc-conjecture.
\newblock {\em Journal of Number Theory}, 30(2):226--237, 1988.

\bibitem[Was98]{Was98}
Lawrence~C. Washington.
\newblock {$p$}-adic {$L$}-functions and sums of powers.
\newblock {\em J. Number Theory}, 69(1):50--61, 1998.

\bibitem[Yam10]{Yam10}
Go~Yamashita.
\newblock Bounds for the dimensions of p-adic multiple l-value spaces.
\newblock {\em Documenta Mathematica}, Extra Volume Suslin:687–723, 2010.

\bibitem[Yas16]{Yas16}
Seidai Yasuda.
\newblock Finite real multiple zeta values generate the whole space {$Z$}.
\newblock {\em International Journal of Number Theory}, 12(03):787--812, 2016.

\bibitem[Zha08]{Zha08}
Jianqiang Zhao.
\newblock Wolstenholme type theorem for multiple harmonic sums.
\newblock {\em Int. J. Number Theory}, 4(1):73--106, 2008.

\end{thebibliography}
\end{document}